\documentclass[english]{article}
\title{On the Well-posedness of a Class of Non-Autonomous SPDEs: An Operator-Theoretical Perspective}
\author{Rainer Picard\footnote{TU Dresden, Fachrichtung Mathematik, Insitut f{\"u}r Analysis, 01062 Dresden, Germany, rainer.oicard@tu-dresden.de}, Sascha Trostorff\footnote{TU Dresden, Fachrichtung Mathematik, Insitut f{\"u}r Analysis, 01062 Dresden, Germany, sascha.trostorff@tu-dresden.de}\phantom{a}and Marcus Waurick\footnote{University of Strathclyde, Livingstone Tower, Department of Mathematics and Statistics, G1 1XH, GLasgow, Scotland, marcus.waurick@strath.ac.uk}}
\date{Version of \today}

\usepackage[utf8]{inputenc}
\usepackage{mathtools}
\usepackage{babel}
\usepackage{url}
\usepackage{a4wide}
\usepackage{xcolor}
\usepackage{amsmath,amsthm,amssymb,amsxtra}
\usepackage{mathrsfs}
\usepackage{enumitem}

\let\geq\geqslant
\let\leq\leqslant

\allowdisplaybreaks

\newcommand*{\N}{\mathbb{N}}									%natural numbers
\newcommand*{\R}{\mathbb{R}}									%real numbers
\renewcommand*{\P}{\mathbb{P}}	

\renewcommand{\d}{\mathrm{d}}							%probability measure

\DeclareMathAccent{\Circ}{\mathalpha}{operators}{"17}
\newcommand{\interior}[1]{\Circ{#1}}
\DeclareMathOperator*{\dom}{dom}
\DeclareMathOperator*{\ran}{ran}
\DeclareMathOperator*{\kar}{ker}
\DeclareMathOperator{\grad}{grad}
\DeclareMathOperator{\curl}{curl}

\DeclareMathOperator{\Grad}{Grad}
\DeclareMathOperator{\dive}{div}
\DeclareMathOperator{\Dive}{Div}

\numberwithin{equation}{section}

\theoremstyle{plain}
\newtheorem{lem}{Lemma}[section]
\newtheorem{theorem}[lem]{Theorem}
\newtheorem{prop}[lem]{Proposition}
\newtheorem{cor}[lem]{Corollary}

\theoremstyle{definition}
\newtheorem{definition}[lem]{Definition}
\newtheorem{ass}[lem]{Assumption}

\theoremstyle{remark}
\newtheorem{remark}[lem]{Remark}
\newtheorem{ex}[lem]{Example}

\definecolor{rot}{rgb}{1,0,0}

\begin{document}
\maketitle

\begin{abstract}
	We further elaborate on the solvability of stochastic partial differential
equations (SPDEs). We shall discuss non-autonomous partial differential
equations with an abstract realization of the stochastic integral
on the right-hand side. Our approach allows the treatment of equations
with mixed type, where classical solution strategies fail to work.
The approach extends prior observations in {[}S\"u\ss, A. \& Waurick,
M. A Solution Theory for a General Class of SPDEs. \emph{Stochastics
and Partial Differential Equations: Analysis and Computations}, 2017,
5, 278-318{]}, where the respective results were obtained for linear
autonomous equations and (multiplicative) white noise. 
\end{abstract}


\vspace*{0.125cm}


%\section*{Acknowledgments}

\section{Introduction}

In this article we discuss the well-posedness of and causality for
a class of non-autonomous partial differential equations/inclusions
perturbed with multiplicative noise. Our strategy is based on the
rationale outlined in \cite{SW16_SD}. In this reference equations of the following
type were discussed: 
\begin{equation}
\left(\partial_{0}M(\partial_{0}^{-1})+A\right)u=\int_{0}^{\cdot}\sigma(u)dB(s),\label{eq:auto}
\end{equation}
where $B$ is an appropriate (vector-valued) Brownian motion, $\sigma$
is a Lipschitz continuous mapping, $\partial_{0}$ is the time derivative,
$M\colon z\mapsto M(z)\in L(H)$ is an analytic function that allows
for defining $M(\partial_{0}^{-1})$ by means of an appropriate functional
calculus, and $A$ is a skew-selfadjoint operator in a Hilbert space
$H$. It has been shown that many standard stochastic partial differential
equations fit into the framework described by \eqref{eq:auto}.

In fact, the stochastic heat and wave equation with multiplicative
noise are special cases of \eqref{eq:auto}. In particular, it is
also possible to formulate a version of Maxwell's equation with multiplicative
noise.

In this article we will enlarge the admissible class of stochastic
partial differential equations towards non-autonomous or even non-linear
inclusions, which are subject to a stochastic perturbation of the
right-hand side.

In a nutshell, the strategy outlined in \cite{SW16_SD}, that is, a way to solving \eqref{eq:auto}, is to find a Hilbert space that
leads to 
\[
u\mapsto\left(\partial_{0}M(\partial_{0}^{-1})+A\right)^{-1}\int_{0}^{\cdot}\sigma(u)dB(s),
\]
being a strict contraction. In this exposition, an adapted result can be found in Theorem~\ref{thm:linst} (see also Theorem~\ref{thm:wp_stoch_nonlin}). In order to obtain the main result in \cite{SW16_SD}, a key observation is that the operator $\left(\partial_{0}M(\partial_{0}^{-1})+A\right)^{-1}$
is \emph{causal}, which implies that in the fixed point iteration
predictable processes are mapped to predictable processes, see also Theorem~\ref{thm:causal,adap}. The above mapping becomes a strict contraction as the Hilbert space setting is formulated in such a way that the Lipschitz constant of $u\mapsto\int_{0}^{\cdot}\sigma(u)dB(s)$ can be made
arbitrarily small, see Proposition~\ref{prop:int_evo} here.

In comparison to \cite{SW16_SD}, we shall not elaborate so much on
the classical notions of solving stochastic partial differential equations,
but rather refer the reader instead to standard monographs such as \cite{prevotroeckner,rozowskii,walsh}.

We shall describe the plan of this note next. After having exemplified typical applications of the rational developed in this manuscript, we establish -- as a another
key ingredient -- the time derivative as a normal and \emph{continuously
invertible} operator in exponentially weighted Hilbert spaces. We recall the notion of evolutionary mappings and causality and draw some interconnections of these concepts. The basic fixed point theorem to be applied to stochastic partial differential equations can be found in Theorem~\ref{thm:cmp}.

 Afterwards, in Section \ref{sec:deterministic}, we recall the essentials of the deterministic solution theory for
non-autonomous equations. We shall also mention a non-linear variant
of the solution theory at hand so that non-linear stochastic partial
differential inclusions can like-wise be considered. 

Section \ref{sec:stochastics} is devoted to the discussion of stochastic evolutionary equations. In Section~\ref{sec:SI}, we will set the stage for the probabilistic solution theory and rephrase the description of stochastic integration as outlined in \cite{Metivier1980}.  The solution theory for non-autonomous stochastic evolutionary equations is provided in Section~\ref{sec:stsee}. More precisely, Theorems~\ref{thm:linst} and~\ref{thm:wp_stoch_nonlin} are the main contributions of this manuscript substantially extending the main result of \cite{SW16_SD}. In Section~\ref{sec:wave}, we conclude this article by providing some examples, which might
be difficult -- if not impossible -- to treat with a more classical
approach.

\section{A glimpse on some particular results}

In order to describe particular applications of the well-posedness results in this manuscript, we need to introduce some operators from vector analysis realised as certain unbounded operators in Hilbert spaces.

Throughout, let $\Omega\subseteq \mathbb{R}^d$ be an open set for some integer $d>0$.

\begin{definition} We define
\begin{align*}
  & \interior{\grad}\colon H_0^1(\Omega)\subseteq L^2(\Omega) \to L^2(\Omega)^d, \phi\mapsto (\partial_i \phi)_{i\in\{1,\ldots,d\}},\\
 &   \interior{\Grad} \colon H_0^1(\Omega)^d \subseteq L^2(\Omega)^d \to L^2_{\textnormal{sym}}(\Omega)^{d\times d}, \Phi\mapsto \frac12(\nabla\Phi+\nabla\Phi^T),\\ &
   \curl\colon H(\curl,\Omega)\subseteq L^2(\Omega)^3 \to L^2(\Omega)^3, \Psi\mapsto \Big(\sum_{j,k\in\{1,2,3\}} \varepsilon_{ijk}\partial_j\Psi_k\Big)_{i\in\{1,2,3\}},
\end{align*}
where $H(\curl,\Omega)$ is the space of $L^2$-vector fields with distributional $\curl$ still being in $L^2$. The space $L^2_{\textnormal{sym}}(\Omega)^{d\times d}$ denotes the set of symmetric $d$-by-$d$ matrices with entries in $L^2(\Omega)$. 

We also put $\interior{\curl}\coloneqq \curl^*$, $\dive\coloneqq -\interior{\grad}^*$, and $\Dive\coloneqq -\interior{\Grad}^*$. We note that for $\Omega$ with sufficiently smooth boundary belonging to the domain of $\interior{\curl}$ corresponds to $H(\curl,\Omega)$-vector field with vanishing tangential component at the boundary. The definitions presented, however, do not require any regularity of the boundary.
\end{definition}

In all the examples to come, we assume that $G$ is a separable Hilbert space and that $L\subseteq L(G,H)$ satisfies the assumptions in \ref{ass:stochastics} below for either of the choices \[H \in \{ L^2(\Omega), L^2(\Omega)^3\oplus L^2(\Omega)^3, L^2(\Omega)^d\},\] which will be clear from the context. Moreover, we assume $\sigma\colon H\to L$ be Lipschitz continuous with $\sigma(0)=0$. In order to have a concrete example at hand, we shall choose $X=W$ to be the Wiener process with values in $G$ satisfying $W(t)=0$ for all $t\leq 0$. Note that in this case $\mathcal{F}$ is chosen to be the natural filtration of $W$ as in Example \ref{ex:sta}(a), which leads to $\alpha=\lambda\otimes \mathbb{P}$ ($\lambda$ denoting the Lebesgue measure on $\mathbb{R}_{\geq 0}$). For the definition of $\mathcal{I}_\nu^{W,\alpha}$ we refer to Section \ref{sec:SI}; also consult Corollary \ref{cor:inv_evo} for $\mathcal{I}^{W,\alpha}\circ\tilde{\sigma}$.

Also we refer to Section \ref{sec:tdem} for a definition of $\partial_{0,\nu}$ and to Definition \ref{def:pr}(c) for a definition of $L_{\nu,\textnormal{pr}}^2(\mathbb{R};L^2(\mathbb{R};H))$. 

\subsection{Standard linear examples}\label{sec:sle}

We start out with a first order formulation of the heat equation. Let $a\colon \mathbb{R}\to L(L^2(\Omega)^d)$ be bounded, Lipschitz continuous with $a(t)=a(t)^*\geq c$ for all $t\in\mathbb{R}$ and some $c>0$; denote $b(t)\coloneqq a(t)^{-1}$ and by $b'$ the weak derivative of $b$.

\begin{theorem}\label{thm:h} There exists $\nu>0$ such that for all $f\in L_{\nu,\textnormal{pr}}^2(\mathbb{R};L^2(\mathbb{P};L^2(\Omega)))$ we find a uniquely determined $u_f \in L_{\nu,\textnormal{pr}}^2(\mathbb{R};L^2(\mathbb{P};L^2(\Omega))$ and $q_f \in L_{\nu,\textnormal{pr}}^2(\mathbb{R};L^2(\mathbb{P};L^2(\Omega)^d)$ such that
\begin{multline*}
 \left(  \overline{ \partial_{0,\nu}\begin{pmatrix} 0 & 0 \\ 0 & b(\mathrm{m})
   \end{pmatrix} + \begin{pmatrix} 1 & 0 \\ 0 & -b'(\mathrm{m}) \end{pmatrix} + \begin{pmatrix} 0 & \dive \\ \interior{\grad} & 0 \end{pmatrix} }   \right)
   \begin{pmatrix} u_f \\ q_f \end{pmatrix} \\= \begin{pmatrix} \partial_{0,\nu}^{-1} f + \mathcal{I}^{W,\alpha}\circ \tilde{\sigma}(u_f) \\ 0 \end{pmatrix}.
\end{multline*}
\end{theorem}
\begin{proof} This is a special case of Theorem \ref{thm:wh} for $P=0$, $C=\interior{\grad}$.
\end{proof}
We shall formally rewrite the equation satisfied by $(u_f,q_f)$ in Theorem \ref{thm:h}. We compute using the second equation,
\[
   b(\mathrm{m})\partial_{0,\nu}q_f =  \partial_{0,\nu} b(\mathrm{m})q_f-b'(\mathrm{m})q_f = -\interior{\grad} u_f.
\]
Hence,
\[
    \partial_{0,\nu} q_f = - a(\mathrm{m})\interior{\grad} u_f.
\]
This, in turn, leads to, using the first equation,
\[
   \partial_{0,\nu}^{-1} f + \mathcal{I}^{W,\alpha}\circ \tilde{\sigma}(u_f) =  u_f +\dive q_f = u_f - \partial_{0,\nu}^{-1}\dive a(\mathrm{m})\interior{\grad} u_f,
\]
or,
\[
   \partial_{0,\nu} u_f - \dive a(\mathrm{m})\interior{\grad} u_f = f + \partial_{0,\nu}\mathcal{I}^{W,\alpha}\circ \tilde{\sigma}(u_f),
\]
which is the stochastic heat equation.

Quite similarly, one can deal with the stochastic wave equation. We shall, however, discuss a different hyperbolic type example next -- the stochastic Maxwell's equations. For this, we let $\varepsilon, \mu,\eta \colon \mathbb{R}\to L(L^2(\Omega)^3)$ be bounded, Lipschitz continuous with $\varepsilon(t)=\varepsilon(t)^*,\mu(t)=\mu(t)^*\geq c$ for all $t\in\mathbb{R}$ and some $c\geq 0$. In this case $\varepsilon,\mu,\eta$ describe the dielectricity, the magnetic permeability, and the electric conductivity, respectively.

\begin{theorem}\label{thm:max} There exists $\nu>0$ such that for all $(J,K)\in L_{\nu,\textnormal{pr}}^2(\mathbb{R};L^2(\mathbb{P};L^2(\Omega)^3\oplus L^2(\Omega)^3))$ there exists uniquely determined $(E,H)\in L_{\nu,\textnormal{pr}}^2(\mathbb{R};L^2(\mathbb{P};L^2(\Omega)^3\oplus L^2(\Omega)^3))$ such that
\[
 \left(  \overline{ \partial_{0,\nu}\begin{pmatrix} \varepsilon(\mathrm{m}) & 0 \\ 0 & \mu(\mathrm{m})
   \end{pmatrix} + \begin{pmatrix} \eta(\mathrm{m}) & 0 \\ 0 & 0 \end{pmatrix} + \begin{pmatrix} 0 & -\curl \\ \interior{\curl} & 0 \end{pmatrix} }   \right)
   \begin{pmatrix}E\\ H \end{pmatrix} = \begin{pmatrix} J \\ K \end{pmatrix} + \mathcal{I}^{W,\alpha}\circ \tilde{\sigma}((E,H)).
\]
\end{theorem}
\begin{proof}
 The result follows upon applying Theorem \ref{thm:linst}, which in turn prerequisites the validity of Assumption \ref{ass:1} under the setting: 
 \begin{align*}
   & H= L^2(\Omega)^3\oplus L^2(\Omega)^3, &  & A=\begin{pmatrix} 0 & -\curl \\ \interior{\curl} & 0 \end{pmatrix},\\
   &  \mathcal{M}=\begin{pmatrix} \varepsilon(\mathrm{m}) & 0 \\ 0 & \mu(\mathrm{m})
   \end{pmatrix}, & & \mathcal{M}'=\begin{pmatrix} \varepsilon'(\mathrm{m}) & 0 \\ 0 & \mu'(\mathrm{m})
   \end{pmatrix}, \\
   & \mathcal{N}=  \begin{pmatrix} \eta(\mathrm{m}) & 0 \\ 0 & 0 \end{pmatrix}.
 \end{align*}
 By Lemma \ref{lem:Asks} applied to $C=-\interior{\curl}$, $A$ is m-accretive. It is also elementary to see that $\mathcal{M}$, $\mathcal{M}'$ and $\mathcal{N}$ satisfy the positive definiteness conditions, we refer to Lemma~\ref{lem:pdwh} for a similar argument. This concludes the proof.
\end{proof}

\subsection{A nonlinear example}\label{sec:nex}

We conclude this examples section with a stochastic variant of the equations of viscoplasticity with internal variables. We refer to \cite{Alber1998,Trostorff2014} for a deterministic model. 

In the stochastic setting discussed here, we shall assume that the equations governing the displacement field $u$ are stochastically perturbed and that the nonlinearity is slightly different from the equations discussed in \cite{Alber1998,Trostorff2014}. Let $N\in\mathbb{N}$.  The system to be studied reads
\begin{align}
  \partial_{0,\nu}^2 Ru - \Dive T &= f+ \partial_{0,\nu}\mathcal{I}^{W,\alpha}\circ \tilde{\sigma}(u),\label{eq:v1}\\
  T &= D(\interior{\Grad} u - Bz),\label{eq:v2} \\
  (B^*\partial_{0,\nu}^{-1}T-L\partial_{0,\nu}^{-1}z,z)& \in g,\label{eq:v3}
\end{align}
where $R\colon \mathbb{R}\to L(L^2(\Omega)^d)$, $D\colon \mathbb{R}\to L(L^2_{\textnormal{sym}}(\Omega)^{d\times d})$, $L\colon \mathbb{R}\to L(L^2(\Omega)^N)$ are bounded, Lipschitz continuous, $R(t)=R(t)^\ast$,$L(t)=L(t)^*$, $D(t)=D(t)^*$ with $R(t)\geq c$, $D(t)\geq c$ and $L(t)\geq c$ for all $t\in \mathbb{R}$ and some $c>0$. $D$ is the elasticity tensor, $f$ is a given volume force, $B\in L(L^2(\Omega)^N,L^2_{\textnormal{sym}}(\Omega)^{d\times d})$ describes the inelastic part $e_p=Bz$ of the strain tensor $e=\interior{\Grad} u$; $g\subseteq L^2(\Omega)^N\oplus L^2(\Omega)^N$ is a maximal monotone relation with $(0,0)\in g$. The unknowns of the above model are the displacement $u$, the stress tensor $T$, and the vector of internal variables $z$, where the latter assumes values in $L^2(\Omega)^N$. 

We reformulate the system \eqref{eq:v1}--\eqref{eq:v3}. For this, we introduce
\[
  \hat{T}\coloneqq \partial_{0,\nu}^{-1} T \text{ and } w\coloneqq B^*\hat{T}-L(\mathrm{m})\partial_{0,\nu}^{-1}z.
\]
Then \eqref{eq:v1} reads
\[
   \partial_{0,\nu} R(\mathrm{m}) u - \Dive \hat{T}=\partial_{0,\nu}^{-1}f+\mathcal{I}^{W,\alpha}\circ \tilde{\sigma}(u).
\]
Furthermore, \eqref{eq:v3} becomes
\[
   (w,\partial_{0,\nu}L^{-1}(\mathrm{m})(B^*\hat{T} -w)) \in g
\]and \eqref{eq:v2} yields
\[
   \partial_{0,\nu}  D^{-1}(\mathrm{m})\hat{T}- (D^{-1})'(\mathrm{m})\hat{T} =  D^{-1}(\mathrm{m}) \partial_{0,\nu}\hat{T} =\interior{\Grad} u - \partial_{0,\nu}BL^{-1}(\mathrm{m})(B^*\hat{T}-w).
\]
Altogether, we obtain
\begin{equation*}
   \left( \begin{pmatrix} u \\ w \\ \hat{T} \end{pmatrix}, \begin{pmatrix} \partial_{0,\nu}^{-1}f+\mathcal{I}^{W,\alpha}\circ \tilde{\sigma}(u) \\ 0 \\ 0 \end{pmatrix} \right) \in  \partial_{0,\nu} \mathcal{M}+\mathcal{N}+A.
\end{equation*}
with
\begin{align*}
  \mathcal{M} &= \begin{pmatrix} R(\mathrm{m}) & 0 & 0 \\ 0 & L^{-1}(\mathrm{m}) & -L^{-1}(\mathrm{m})B^\ast \\ 0 & -BL^{-1}(\mathrm{m}) & D^{-1}(\mathrm{m})+BL^{-1}(\mathrm{m})B^\ast 
   \end{pmatrix}\\ \mathcal{N}&= \begin{pmatrix} 0 & 0 & 0 \\ 0 & 0 & 0 \\ 0 & 0 & 
   -(D^{-1})'(\mathrm{m})   
   \end{pmatrix} \text{ and } \\
   A&= \begin{pmatrix}
        0& 0& -\Dive\\
        0 & g& 0 \\
        -\interior{\Grad} & 0 &0
      \end{pmatrix}.
\end{align*}
By applying a symmetric Gauss-step, it is not difficult to show that the operators $\mathcal{M},\mathcal{N}$ satisfy the assumptions stated in Theorem \ref{thm:wp_stoch_nonlin} with $H=L^2(\Omega)^d \oplus L^2(\Omega)^N\oplus L^2_{\textnormal{sym}}(\Omega)^{d\times d}$. Moreover, note that by \cite[p. 64]{Trostorff2014}, also the relation $A$ satisfies the assumptions in Theorem \ref{thm:wp_stoch_nonlin}. Thus, we have shown the following result:
\begin{theorem} Let $\mathcal{M},\mathcal{N}$ and $A$ as above. Then there exists $\nu>0$ such that for all $f\in L_{\nu,\textnormal{pr}}^2(\mathbb{R};L^2(\mathbb{P};L^2(\Omega)^d))$ there exists a unique $\left(  u ,  w, \hat{T}\right) \in L_{\nu,\textnormal{pr}}^2(\mathbb{R};L^2(\mathbb{P};L^2(\Omega)^d \oplus L^2(\Omega)^N\oplus L^2_{\textnormal{sym}}(\Omega)^{d\times d}))$ satisfying
\[
   \overline{\partial_{0,\nu}\mathcal{M}+\mathcal{N}+A}\ni \left(\left(  u ,  w, \hat{T}\right), (\partial_{0,\nu}^{-1}f+\mathcal{I}^{W,\alpha}\circ \tilde{\sigma}(u),0,0)\right).
\]
\end{theorem}
We shall now develop the theory in order to properly justify the above results.
\section{The time derivative and evolutionary mappings}

\label{sec:tdem}

Let $H$ be a Banach space. For $\nu\in\mathbb{R}$
we define 
\[
L_{\nu}^{2}(\mathbb{R};H)\coloneqq\{f\in L_{\textnormal{loc}}^{2}(\mathbb{R};H);\int_{\mathbb{R}}|f(t)|_{H}^{2}\exp(-2\nu t)\, \d t<\infty\}
\]
endowed with the obvious norm. It is easy to see that $L_{\nu}^{2}(\mathbb{R};H)$
is a Banach space, as well. Specializing to $H$ being a Hilbert space, we denote by $H_{\nu}^{1}(\mathbb{R};H)$ the
Sobolev space of once weakly differentiable functions with derivative
in $L_{\nu}^{2}(\mathbb{R};H)$. We obtain (see \cite[Section 2]{KPSTW14_OD}),
that 

\begin{align*}
\partial_{0,\nu}\colon H_{\nu}^{1}(\mathbb{R};H)\subseteq L_{\nu}^{2}(\mathbb{R};H) & \to L_{\nu}^{2}(\mathbb{R};H)\\
f & \mapsto f'
\end{align*}

is a densely defined, closed and normal linear operator. Moreover, we have
$\partial_{0,\nu}^{*}=-\partial_{0,\nu}+2\nu$. In applications to
be discussed later on, $\partial_{0,\nu}$ will be our realization
of the time derivative for $\nu>0$ `large enough'. Note that for
$\nu>0$, we obtain that $\partial_{0,\nu}$ is continuously invertible
with 
\[
\partial_{0,\nu}^{-1}f(t)=\int_{-\infty}^{t}f(\tau)\,\d\tau,
\]
where the integral is well-defined for all $f\in L_{\nu}^{2}(\mathbb{R};H)$
in the Bochner sense and we have $\|\partial_{0,\nu}^{-1}\|=1/\nu$,
see also \cite[Corollary 2.5]{KPSTW14_OD}.

For the treatment of evolutionary equations with non-autonomous coefficients,
we will need the notion of evolutionary mappings. In fact, also in
the discussion of stochastic partial differential equations, this
notion proved useful for the abstract description of the stochastic
integral.

\begin{definition}\label{def:evolutionary} Let $H,G$ be Banach spaces,
$\nu>0$. Let 
\[
F:\dom(F)\subseteq\bigcap_{\mu\geq\nu}L_{\mu}^{2}(\R;H)\to\bigcap_{\mu\geq\nu}L_{\mu}^{2}(\R;G),
\]
where $\dom(F)$ is supposed to be a vector space. We call $F$ \emph{evolutionary
(at $\nu$)}, if for all $\mu\geq\nu$, $F$ satisfies the following
properties 

\begin{enumerate}[label=(\roman{enumi}),ref=(\roman{enumi})] 
\item $F$ is Lipschitz continuous as a mapping 
\[
F_{0,\mu}\colon\dom(F)\subseteq L_{\mu}^{2}(\R;H)\to L_{\mu}^{2}(\R;G),\,\phi\mapsto F(\phi),
\]

\item $\|F\|_{\textrm{ev},\textrm{Lip}}\coloneqq\limsup_{\mu\to\infty}\|F_{\mu}\|_{\textrm{Lip}}<\infty$,
with $F_{\mu}\coloneqq\overline{F_{0,\mu}}$ denoting the Lipschitz continuous extension of $F$.\end{enumerate} 

The non-negative number $\|F\|_{\textrm{ev},\textrm{Lip}}$ is called
the \emph{the eventual Lipschitz constant of $F$}. We denote 
\[
L_{\textnormal{ev},\nu}(H,G)\coloneqq\{F;F\text{ evolutionary at }\nu\},\quad L_{\textnormal{ev},\nu}(H)\coloneqq L_{\textnormal{ev},\nu}(H,H).
\]
If, in addition, $F_{\mu}$ leaves $\dom(F_{\mu})=\overline{\dom(F)}^{L_{\mu}^2}$
invariant ($\mu\geq\nu$), then we call $F$ \emph{invariant evolutionary
(at $\nu$)}. A mapping $F$, which is evolutionary at $\nu$, is
called \emph{densely defined}, if, for all $\mu\geq\nu$, $\dom(F)\subseteq L_{\mu}^2(\mathbb{R};H)$
is dense. \end{definition}

Next, we introduce the concept of causality, as it has been introduced
in \cite{W15_CB} as a particular concept for (nonlinear) mappings
in Banach spaces. In the applications to follow, we will mainly focus on
Lipschitz continuous Hilbert space valued mappings, see also
\cite[Definition 2.2.2 and Remark 2.2.3]{W16_H}.

\begin{definition} (a) Let $H$ be a Banach space. A family $R=(R_{t})_{t\in \mathbb{R}}$
is called \emph{resolution of the identity (in $H$)}, if for all
$t\in\mathbb{R}$ 
\[
R_{t}=R_{t}^{2} \in L(H)\text{ and }2\lim_{t\to\pm\infty}R_{t}=1\pm1,
\]
where the limit is in the strong operator topology of $L(H)$. The
pair $(H,R)$ is called resolution space.

(b) Let $(H,R),(K,Q)$ be resolution spaces and let $F\colon\dom(F)\subseteq H\to K$,
$D\subseteq K'$. We call $F$ \emph{causal on $D$}, if for all $r>0$,
$t\in\mathbb{R}$, $\phi\in D$, the mapping 
\[
(B_{F}(0,r),|R_{t}(\cdot-\cdot)|)\to(K,|\langle Q_{t}(\cdot-\cdot),\phi\rangle|),x\mapsto F(x)
\]
is Lipschitz continuous, where $B_{F}(0,r)\coloneqq\{x\in\dom(F);|x|^{2}+|F(x)|^{2}<r\}$;
if $D=K'$, we call $F$ \emph{causal}. \end{definition}

\begin{prop}[{{\cite[Theorem 1.7]{W15_CB} and \cite[Theorem 2.2.4]{W16_H}}}] \label{prop:causal} Let $(H,R)$, $(K,Q)$ be resolution
spaces, $F\colon\dom(F)\subseteq H\to K$ densely defined and Lipschitz
continuous, $D\subseteq K'$ separating for $K$. Then the following
conditions are equivalent: 
\begin{enumerate}
\item[(i)] $F$ is causal on $D$, 
\item[(ii)] $\overline{F}$ is causal, 
\item[(iii)] for all $t\in\mathbb{R}$, 
\[
(\dom(F),|R_{t}(\cdot-\cdot)|\to(K,|Q_{t}(\cdot-\cdot)|),x\mapsto F(x)
\]
is Lipschitz continuous. 
\item[(iv)] for all $t\in\mathbb{R}$, we have $Q_{t}\circ\overline{F}=Q_{t}\circ\overline{F}\circ R_{t}$. 
\end{enumerate}
\end{prop} \begin{proof} The implication (ii)$\Rightarrow$(i) is
trivial, both the implications (iii)$\Rightarrow$(ii) and (iv)$\Rightarrow$(iii)
are easy to obtain. Thus, it suffices to prove that (i) is sufficient
for (iv). For this, let $t\in\mathbb{R}$ and $\phi\in D$. For $\psi\in\dom(F)$
we find $\psi_{n}\in\dom(F)$ such that $\psi_{n}\to R_{t}\psi$ in
$H$ as $n\to\infty$. By the boundedness of $(\psi_{n})_{n}$ in
$H$, and by causality of $F$ on $D$, we find $C\geq0$ such that
for all $n\in\mathbb{N}$, we obtain 
\[
|\langle Q_{t}F(\psi)-Q_{t}F(\psi_{n}),\phi\rangle|\leq C|R_{t}\psi-R_{t}\psi_{n}|.
\]
Letting $n\to\infty$ in the latter inequality and using that $R_{t}^{2}=R_{t}$,
we deduce that 
\[
|\langle Q_{t}F(\psi)-Q_{t}\overline{F}(R_{t}\psi),\phi\rangle|\leq C|R_{t}\psi-R_{t}^{2}\psi|=0.
\]
Thus, since $D$ is separating for $K$, we infer 
\[
Q_{t}F(\psi)=Q_{t}\overline{F}(R_{t}\psi)\quad(\psi\in\dom(F)).
\]
By continuity, we obtain $Q_{t}\circ\overline{F}=Q_{t}\circ\overline{F}\circ R_{t}$.
\end{proof}

We recall a variant of \cite[Lemma 2.13]{SW16_SD}. As the assumptions
vary slightly from the ones used in \cite[Lemma 2.13]{SW16_SD}, we
carry out the proof.

\begin{lem}\label{lem:indep_of_evo} Let $F$ be evolutionary at
$\nu>0$. Assume that $\dom(F)\cap\dom(F\chi_{(-\infty,a]})$ is dense
in $\dom(F)$ with respect to the $L_{\mu}^{2}(\mathbb{R};H)$-norm
for all $a\in\R$ and $\mu\geq\nu$. Then $F_{\eta}|_{\dom(F_{\eta})\cap\dom(F_{\mu})}=F_{\mu}|_{\dom(F_{\eta})\cap\dom(F_{\mu})}$
for all $\eta\geq\mu\geq\nu$. \end{lem} \begin{proof} Let $\phi\in\dom(F_{\eta})\cap\dom(F_{\mu})$.
By assumption, we can choose a sequence $(\phi_{n})_{n}$ in $\dom(F)\cap\dom(F\chi_{(-\infty,a]})$
such that $\phi_{n}\to\phi$ in $L_{\mu}^{2}(\mathbb{R};H)$ as $n\to\infty$.
Moreover, we deduce that $\dom(F)\ni\chi_{(-\infty,a]}\phi_{n}\to\chi_{(-\infty,a]}\phi$
in both $L_{\eta}^{2}(\mathbb{R};H)$ and $L_{\mu}^{2}(\mathbb{R};H)$
as $n\to\infty$. In particular, we obtain $\chi_{(-\infty,a]}\phi\in\dom(F_{\mu})\cap\dom(F_{\eta})$
and 
\begin{align*}
F_{\mu}(\chi_{(-\infty,a]}\phi) & =\lim_{n\to\infty}F_{\mu}(\chi_{(-\infty,a]}\phi_{n})\\
 & =\lim_{n\to\infty}F(\chi_{(-\infty,a]}\phi_{n})\\
 & =\lim_{n\to\infty}F_{\eta}(\chi_{(-\infty,a]}\phi_{n})=F_{\eta}(\chi_{(-\infty,a]}\phi).
\end{align*}
Next, we note that $\chi_{(-\infty,a]}\phi\to\phi$ as $a\to\infty$
in $L_{\eta}^{2}(\mathbb{R};H)$ and $L_{\mu}^{2}(\mathbb{R};H)$
since $\phi\in L_{\eta}^{2}(\mathbb{R};H)\cap L_{\mu}^{2}(\mathbb{R};H)$.
Hence, $F_{\mu}(\phi)=F_{\eta}(\phi)$. \end{proof}

We shall further point out another consequence of evolutionarity and
the condition on the domain in the previous result. In fact, this
is a combination of the arguments used for \cite[Theorem 4.5]{KPSTW14_OD}
and \cite[Remark 2.1.5]{W16_H}. For this, from now on and throughout
the whole manuscript, we shall use $R_{t}=Q_{t}=\chi_{(-\infty,t]}$
as the standard resolution of the identity, and thus $(L_{\nu}^{2}(\mathbb{R};H),(\chi_{(-\infty,t]})_{t})$
as resolution space.

\begin{lem}\label{lem:criteriona_for_causal} Let $F$ be evolutionary
at $\nu>0$. Assume that $\dom(F)\cap\dom(F\chi_{(-\infty,a]})$ is
dense in $\dom(F)$ with respect to the $L_{\mu}^{2}(\mathbb{R};H)$-norm
for all $a\in\R$ and $\mu\geq\nu$. Then $F_{\mu}$ is causal for
all $\mu\geq\nu$. \end{lem} \begin{proof} Let $\mu\geq\nu$. We
apply Proposition \ref{prop:causal} and prove $Q_{t}\circ F=Q_{t}\circ F\circ Q_{t}$
for all $t\in\mathbb{R}$. Note that this implies (iv) in Proposition
\ref{prop:causal} as both the left- and the right-hand side are densely
defined in $\dom(F_{\mu})$. So, let $t\in\mathbb{R}$, $\phi\in\mathring{C}_{\infty}(\mathbb{R};G')$
and $\psi\in\dom(F)\cap\dom(F\chi_{(-\infty,t]})$. We compute for
$\eta\geq\mu$ 
\begin{align*}
& |\langle Q_{t}(F(\psi)-F(Q_{t}\psi)),\phi\rangle_{{L^{2}(\R;G)},L^2(\mathbb{R};G')}| \\ & =|\langle(F(\psi)-F(Q_{t}\psi)),Q_{t}\phi\rangle_{{L^{2}(\R;G)},L^2(\mathbb{R};G')}|\\
 & \leq\|{F_{\eta}}\|_{\textnormal{Lip}}\|\psi-Q_{t}\psi\|_{L_{\eta}^{2}}\|Q_{t}{\phi}\|_{L^{2}}\exp(\eta t).
\end{align*}
We compute further 
\begin{align*}
\|\psi-Q_{t}\psi\|_{L_{\eta}^{2}}^{2}\exp(2\eta t) & =\intop_{\R}|\psi(s)(1-\chi_{(-\infty,t]}(s))|\exp(-2\eta(s-t))\,\d s\\
 & =\intop_{\R}|\psi(s+t)(1-\chi_{(-\infty,t]}(s+t))|\exp(-2\eta s)\,\d s\\
 & =\intop_{0}^{\infty}|\psi(s+t)|\exp(-2\eta s)\,\d s\to0\quad(\eta\to0),
\end{align*}
and thus, we deduce that $Q_{t}(F(\psi)-F(Q_{t}\psi))=0$, as desired.
\end{proof}

We conclude this section with a perturbation result, which we need
for a solution theory for non-autonomous stochastic partial differential
equations.

\begin{theorem}[{{{see also \cite[Corollary 2.15]{SW16_SD}}}}]\label{thm:cmp}
Let $H$ be a Banach space, $\nu>0$, $S,F\in L_{\textnormal{ev},\nu}(H)$,
$F$ invariant evolutionary. Let $S$ be densely defined, 
$\|S\|_{\textnormal{ev},\textnormal{Lip}}\|F\|_{\textnormal{ev},\textnormal{Lip}}<1$,
and $S_{\mu}[\dom(F_{\mu})]\subseteq\dom(F_{\mu})$ for all $\mu\geq\nu$.
Then for all $f\in\dom(F_{\mu})$ the mapping 
\begin{align*}
\Phi_{\mu}(f)\colon\dom(F_{\mu}) & \to\dom(F_{\mu})\\
u & \mapsto S_{\mu}f+S_{\mu}(F_{\mu}(u))
\end{align*}
admits a unique fixed point $u_{f}$ as long as $\mu\geq\nu$ is large
enough, that is, a unique solution $u_{f}$ of the problem 
\begin{equation}
u_{f}-S_{\mu}(F_{\mu}(u_{f}))=S_{\mu}(f).\label{eq:pert_prob}
\end{equation}
The mapping $f\mapsto u_{f}$ is evolutionary. If $S$ and $F$ are
causal, then so is $f\mapsto u_{f}$. If $\dom(F)\cap\dom(F\chi_{(-\infty,a]})$
is dense in $\dom(F)$ with respect to $L_{\mu}^{2}(\mathbb{R};H)$
for all sufficiently large $\mu$, then $f\mapsto u_{f}$ does not
depend on $\mu$ in the sense of Lemma \ref{lem:indep_of_evo}. \end{theorem}
\begin{proof} Let $\mu\geq\nu$ such that $\|S_{\mu}\|_{\textnormal{Lip}}\|F_{\mu}\|_{\textnormal{Lip}}<1$.
Then it is easy to see that $\Phi_{\mu}(f)$ defines a strict contraction.
By standard a posteriori estimates, we deduce that we find $C\geq0$
such that $\|f\mapsto u_{f}\|_{\textnormal{Lip}}\leq1/(1-\|S\|_{\textnormal{ev},\textnormal{Lip}}\|F\|_{\textnormal{ev},\textnormal{Lip}})+C$.

It remains to prove causality of the fixed point mapping. For this
it suffices to observe that $\Phi_{\mu}(f)$ is causal. This, however,
follows from the fact that composition of causal mappings is still
causal.

The independence of $\mu$ is a consequence of Lemma \ref{lem:indep_of_evo}.
\end{proof}

In applications, the mapping $S_{\mu}$ will be the solution operator
of an abstract deterministic partial differential equation and thus,
the solution $u_{f}$ in \eqref{eq:pert_prob} turns out to be the
solution of this deterministic PDE perturbed by an additional mapping
$F_{\mu}$, which will be our stochastic integral operator.

\section{The deterministic solution theory}

\label{sec:deterministic}

In this section we will review the solution theory for a class of
(non-autonomous) linear partial differential equations which has its
roots in the autonomous version presented in \cite{PicPhy}. Later
on, this has been generalized to non-autonomous or non-linear equations,
see e.g.~\cite{Trostorff2012,PTWW13_NA,Trostorff2014,W15_NA}. To
keep this article conveniently self-contained, we shall summarize
the well-posedness theorem outlined in \cite[Theorem 3.4.6]{W16_H}.
However, we will also present the main results of \cite{Trostorff2014},
in order to obtain a non-linear variant for stochastic partial differential
equations.

The main hypothesis for the linear case is presented next.

\begin{ass}[{{\cite[Hypothesis 3.4.4]{W16_H}}}]\label{ass:1}
Let $H$ be a Hilbert space, $\nu>0$, $\mathcal{M},\mathcal{M}',\mathcal{N}\in L_{\textnormal{ev},\nu}(H)$.
Assume that $\mathring{C}_{\infty}(\mathbb{R};H)\subseteq\dom(\mathcal{M})\cap\dom(\mathcal{M}')\cap\dom(\mathcal{N})$.
Let $A\colon\dom(A)\subseteq H\to H$ be densely defined and m-accretive.
Assume that 
\begin{align*}
 & \forall\mu\geq\nu\colon\mathcal{M}\partial_{0,\mu}\subseteq\partial_{0,\mu}\mathcal{M}_{\mu}-\mathcal{M}'_{\mu},\\
 & \exists c>0\,\forall\mu\geq\nu,t\in\R\colon\Re\langle Q_{t}\left(\partial_{0,\mu}\mathcal{M}+\mathcal{N}\right)\phi,\phi\rangle_{0,\mu}\geq c\langle\phi,Q_t\phi\rangle_{0,\mu}\\ &\hspace*{8cm}\quad(\phi\in\mathring{C}_{\infty}(\mathbb{R};H)).
\end{align*}
\end{ass}

With the latter set of assumptions, we can show the following well-posedness
theorem covering a large class of linear non-autonomous evolutionary
equations:

\begin{theorem}[{{\cite[Theorem 3.4.6]{W16_H}}}]\label{thm:st}
Impose Assumption \ref{ass:1}. Then the operator 
\[
\mathcal{B}\coloneqq\partial_{0,\nu}\mathcal{M}+\mathcal{N}+A
\]
is densely defined and closable. Moreover, its closure is onto and
continuously invertible in $L_{\mu}^{2}(\mathbb{R};H)$ for all $\mu\geq\nu$.
Furthermore, $\mathcal{S}\coloneqq\mathcal{B}^{-1}$ is evolutionary
at $\nu$, densely defined and causal. \end{theorem}

\begin{lem}\label{lem:pos_def} Let $G$ be a Hilbert space, $B\colon\dom(B)\subseteq G\to G$
a densely defined linear operator. Assume there exists $c>0$ with
the property that 
\begin{equation}
\Re\langle B\phi,\phi\rangle\geq c\langle\phi,\phi\rangle,\label{eq:coerc1}
\end{equation}
as well as 
\begin{equation}
\Re\langle B^{*}\psi,\psi\rangle\geq c\langle\psi,\psi\rangle,\label{eq:coerc2}
\end{equation}
for all $\phi\in\dom(B)$ and $\psi\in\dom(B^{*})$. Then $B$ is
closable and $\overline{B}^{-1}$ exists as an element of $L(G)$,
the space of bounded linear operators on $G$ and $\|\overline{B}^{-1}\|\leq1/c$.
\end{lem}

\begin{proof} Before we come to the proof of the assertion, we need
some preparations. Note that for $\lambda>0$ the operator $1+\lambda B$
is one-to-one by \eqref{eq:coerc1}. Moreover, by \eqref{eq:coerc2}
we infer that its adjoint $1+\lambda B^{\ast}$ is one-to-one, as
well, and hence, $(1+\lambda B)^{-1}$ is densely defined. Again,
\eqref{eq:coerc1} implies that $(1+\lambda B)^{-1}$ is bounded with
norm less than or equal to $1$. Thus, its closure is an element in
$L(G)$ with the same norm.\\
 For $\phi\in\dom(B)$ we obtain 
\begin{equation}
\overline{(1+\lambda B)^{-1}}\phi-\phi=-\overline{(1+\lambda B)^{-1}}\lambda B\phi\to0\quad(\lambda\to0)\label{eq:approx}
\end{equation}
and since $\dom(B)$ is dense and the family $\left(\overline{(1+\lambda B)^{-1}}\right)_{\lambda>0}$
is bounded, we infer that \eqref{eq:approx} holds for each $\phi\in G$.\\
 We now prove the closability of $B$. For doing so, let $(\phi_{n})_{n\in\N}$
in $\dom(B)$ with $\phi_{n}\to0$ and $B\phi_{n}\to y$ for some
$y\in G$ as $n\to\infty$. Thus, we infer for each $\lambda>0$ 
\begin{align*}
\overline{(1+\lambda B)^{-1}}y & =\lim_{n\to\infty}\overline{(1+\lambda B)^{-1}}B\phi_{n}\\
 & =\frac{1}{\lambda}\lim_{n\to\infty}\left(\phi_{n}-\overline{(1+\lambda B)^{-1}}\phi_{n}\right)\\
 & =0,
\end{align*}
and thus, letting $\lambda$ tend to $0$, \eqref{eq:approx} yields
$y=0$, proving that $B$ is closable. Noting that \eqref{eq:coerc1}
and \eqref{eq:coerc2} yield that $B^{-1}$ is a densely defined bounded
linear operator with norm less than or equal to $\frac{1}{c}$, the
assertion follows with $\overline{B^{-1}}=\overline{B}^{-1}$. \end{proof}

The crucial part of the proof of Theorem \ref{thm:st} is to show
that $\mathcal{B}$ has dense range. For this, we will employ the
following lemma.

\begin{lem}\label{lem:comadjo} Let $H$ be a Hilbert space, $D\colon\dom(D)\subseteq H\to H$,
$C\colon\dom(C)\subseteq H\to H$ closed. Assume that $\dom(D)\cap\dom(C)\subseteq H$
is dense. Furthermore, let $(T_{n})_{n}$ in $L(H)$ be such that
$T_{n}\to1$ in the strong operator topology. Moreover, assume that
$T_{n}[\dom(D)]\subseteq\dom(D)$ and $\ran(T_{n})\subseteq\dom(C)$
for each $n\in\N$, as well as 
\begin{align*}
 & [T_{n},C],[T_{n},D]\text{ bounded for all \ensuremath{n\in\N}, }\\
 & \overline{[T_{n},C]},\overline{[T_{n},D]}\to0\ (n\to\infty),
\end{align*}
where the convergence holds in the strong operator topology. Then $(C+D)^{*}=\overline{C^{*}+D^{*}}$.
\end{lem} \begin{proof} Note that $(C+D)^{*}\supseteq\overline{C^{*}+D^{*}}$
is clear. So, let $\phi\in\dom((C+D)^{*})$ and for $n\in\N$ we define
$\phi_{n}\coloneqq T_{n}^{*}\phi$. At first we show that $\phi_{n}\in\dom((C+D)^{*})$.
For this, let $\eta\in\dom(C+D)$. We compute 
\begin{align*}
\langle(C+D)\eta,\phi_{n}\rangle & =\langle(C+D)\eta,T_{n}^{*}\phi\rangle\\
 & =\langle T_{n}(C+D)\eta,\phi\rangle\\
 & =\langle(C+D)T_{n}\eta,\phi\rangle+\langle[T_{n},C]\eta,\phi\rangle+\langle[T_{n},D]\eta,\phi\rangle\\
 & =\langle\eta,T_{n}^{*}(C+D)^{*}\phi\rangle+\langle\eta,[T_{n},C]^{*}\phi\rangle+\langle\eta,[T_{n},D]^{*}\phi\rangle,
\end{align*}
which shows that $(C+D)^{*}\phi_{n}=T_{n}^{*}(C+D)^{*}\phi+[T_{n},C]^{*}\phi+[T_{n},D]^{*}\phi$.
Next, note that $CT_{n}\in L(H)$ by the closed graph theorem and
$\overline{[T_{n},C]}\in L(H)$ by assumption. Hence, $\overline{T_{n}C}=\overline{[T_{n},C]}+CT_{n}\in L(H)$
as well and thus, we deduce that $\overline{T_{n}C}^{*}=C^{*}T_{n}^{*}\in L(H)$.
In particular, we infer that $T_{n}^{*}$ maps into $\dom(C^{*})$.
Hence, $\phi_{n}\in\dom(C^{*})$. Furthermore, for $\eta\in\dom(D)$,
we have that $T_{m}\eta\in\dom(D)\cap\dom(C)$ by assumption. Moreover,
we have 
\[
T_{m}\eta\to\eta,\text{ and }DT_{m}\eta=[D,T_{m}]\eta+T_{m}D\eta\to D\eta\quad(m\to\infty).
\]
Thus, $\dom(C+D)$ is dense in $\dom(D)$ with respect to the graph
norm of $D$. Altogether, we compute for all $\eta\in\dom(C+D)$ 
\begin{align*}
\langle D\eta,\phi_{n}\rangle & =\langle(C+D)\eta,\phi_{n}\rangle-\langle C\eta,\phi_{n}\rangle\\
 & =\langle\eta,(C+D)^{*}\phi_{n}\rangle-\langle\eta,C^{*}\phi_{n}\rangle,
\end{align*}
which proves that $\phi_{n}\in\dom(D^{*})$ and $D^{*}\phi_{n}=(C+D)^{*}\phi_{n}-C^{*}\phi_{n}$.
Hence, 
\[
T_{n}^{*}(C+D)^{*}\phi+[T_{n},C]^{*}\phi+[T_{n},D]^{*}\phi=(C+D)^{*}\phi_{n}=(D^{*}+C^{*})\phi_{n}.
\]
Next, we may let $n\to\infty$ in the latter equality and obtain the
assertion. \end{proof}

\begin{prop}\label{prop:core}Impose Assumption \ref{ass:1}. Then
$\dom(\partial_{0,\mu})$ is a core for $(\partial_{0,\mu}\mathcal{M})^{*}$.
\end{prop} \begin{proof} It suffices to observe that $\dom(\partial_{0,\mu})=\dom(\partial_{0,\mu}^{*})$
and $(\partial_{0,\mu}\mathcal{M})^{*}=\overline{\mathcal{M}^{*}\partial_{0,\mu}^{*}}$.
\end{proof}

\begin{proof}[Proof of Theorem \ref{thm:st}] Let $\mu\geq\nu$.
First of all note that $\mathcal{B}$ is densely defined, since $\mathring{C}_{\infty}(\mathbb{R};\dom(A))\subseteq\dom(\mathcal{B})$.
Moreover, note that 
\begin{equation}
\Re\langle\mathcal{B}\phi,\phi\rangle_{L_{\mu}^{2}}\geq c\langle\phi,\phi\rangle_{L_{\mu}^{2}}\label{eq:st1}
\end{equation}
by Assumption \ref{ass:1}. Next, since $\mathcal{B}$ is densely
defined, we can use \cite[Theorem 4.2.5]{Beyer2007}, to deduce that
$\mathcal{B}$ is closable. Note that inequality \eqref{eq:st1} remains
true for $\phi\in\dom(\overline{\mathcal{B}})$.

We apply Lemma \ref{lem:pos_def} to the operator $\overline{\mathcal{B}}$.
For this, we compute the adjoint of $\mathcal{B}$. With the setting
$C\coloneqq\partial_{0,\mu}\mathcal{M}_{\mu}+\mathcal{N}_{\mu}$,
$D\coloneqq A$ and $T_{n}\coloneqq(1+(1/n)\partial_{0,\mu})^{-1}$,
we employ Lemma \ref{lem:comadjo}. We check the hypothesis of Lemma
\ref{lem:comadjo} next. First of all, note that $T_{n}$ is well-defined
with $\|T_{n}\|\leq1$ and that $T_{n}\to1$ in the strong operator
topology. Clearly, $T_{n}$ leaves $\dom(D)$ invariant and attains
values in $\dom(C)$. Moreover, both the operators 
\[
T_{n}C\subseteq\partial_{0,\mu}T_{n}\mathcal{M}_{\mu}+T_{n}\mathcal{N}_{\mu}\text{ and }CT_{n}=\big(\partial_{0,\mu}\mathcal{M}_{\mu}+\mathcal{N}_{\mu}\big)T_{n}=\big(\mathcal{M}_{\mu}'+\mathcal{M}_{\mu}\partial_{0,\mu}+\mathcal{N}_{\mu}\big)T_{n}
\]
are densely defined and bounded. Thus, so is 
\[
[T_{n},C]=[T_{n},\mathcal{N}_{\mu}]+\frac{1}{n}\partial_{0,\mu}T_{n}\mathcal{M}_{\mu}'T_{n}.
\]
It is not difficult to see that $\overline{[T_{n},C]}\to0$ as $n\to\infty$.
Observe that 
\[
[T_{n},D]\subseteq0.
\]
So that $\overline{[T_{n},D]}=0\to0$ as $n\to\infty$. Thus, by Lemma
\ref{lem:comadjo}, we infer 
\[
\mathcal{B}^{*}=(C+D)^{*}=\overline{C^{*}+D^{*}}=\overline{\left(\partial_{0,\mu}\mathcal{M}_{\mu}+\mathcal{N}_{\mu}\right)^{*}+A^{*}}.
\]
By the boundedness of $\mathcal{N}_{\mu}$, we deduce that $\left(\partial_{0,\mu}\mathcal{M}_{\mu}+\mathcal{N}_{\mu}\right)^{*}=\left(\partial_{0,\mu}\mathcal{M}_{\mu}\right)^{*}+\mathcal{N}_{\mu}^{*}$.
Thus, by Proposition \ref{prop:core}, $\dom(\partial_{0,\mu})$ is
an operator core for $C^{*}$. For $\phi\in\dom(\partial_{0,\mu})\subseteq\dom(C)$
we compute 
\[
\Re\langle C^{*}\phi,\phi\rangle_{\mu}=\Re\langle\phi,C\phi\rangle_{\mu}\geq c\langle\phi,\phi\rangle.
\]
Thus, $\Re\langle C^{*}\phi,\phi\rangle\geq c\langle\phi,\phi\rangle$
for all $\phi\in\dom(C^{*})$. Moreover, since $A$ is densely defined
and m-accretive, $A^{*}$ is accretive, as well, see \cite{Phillips1959}.
Thus, altogether $\Re\langle\mathcal{B}^{*}\phi,\phi\rangle\geq c\langle\phi,\phi\rangle$.
Therefore, Lemma \ref{lem:pos_def} implies that $\mathcal{B}$ is continuously
invertible and has dense range. In particular, we obtain $\mathcal{B}^{-1}$
is densely defined and has operator norm bounded by $1/c$ so that
$\mathcal{B}^{-1}$ is evolutionary at $\nu$.

For $\psi\in\dom(\mathcal{B}^{-1})$, $\phi\coloneqq\mathcal{B}^{-1}\psi$
we furthermore realize that the inequality 
\[
\Re\langle\mathcal{B}\phi,Q_{t}\phi\rangle\geq c\langle Q_{t}\phi,Q_{t}\phi\rangle
\]
implies 
\[
\|Q_{t}\mathcal{B}^{-1}\psi\|\leq\frac{1}{c}\|Q_{t}\psi\|,
\]
which by Proposition \ref{prop:causal} (iii) is sufficient for causality
of $\overline{\mathcal{B}}^{-1}$. \end{proof}

Next, we slightly rephrase the main result of \cite{Trostorff2014}.
There, a well-posedness result for non-autonomous differential inclusions
is stated, where the operator $A$ is replaced by a maximal monotone
relation on a Hilbert space $H$ (for an introduction to maximal monotone
relations on Hilbert spaces we refer to the monograph \cite{Brezis}).
As a trade-off, we need to restrict the class of admissible operators
$\mathcal{M}$ and $\mathcal{N}$:

\begin{theorem}[{{\cite[Theorem 3.4]{Trostorff2014}}}]\label{thm:nonlinear}
Let $H$ be a separable Hilbert space, let $M,N\colon\mathbb{R}\to L(H)$
be strongly measurable and bounded mappings. Assume that $M(t)$ is
selfadjoint for all $t\in\mathbb{R}$, $M$ Lipschitz continuous,
$A\subseteq H\oplus H$ a maximal monotone relation with $(0,0)\in A$.
Moreover, assume that $K\coloneqq\kar(M(t))=\kar(M(0))$ for all $t\in\mathbb{R}$
and that there exists $c>0$ such that for all $t\in\mathbb{R}$ 
\[
\langle M(t)\phi,\phi\rangle\geq c\langle\phi,\phi\rangle,\ \Re\langle N(t)\psi,\psi\rangle\geq c\langle\psi,\psi\rangle
\]
for all $\phi\in K$ and $\psi\in K^{\bot}$.

Then there exists $\nu>0$, $C\geq0$ such that for all $\mu\geq\nu$
\[
\mathcal{S}_{\mu}\coloneqq\overline{\left(\partial_{0,\mu}\mathcal{M}+\mathcal{N}+A\right)}^{-1}\colon L_{\nu}^{2}(\mathbb{R};H)\to L_{\nu}^{2}(\mathbb{R};H)
\]
is Lipschitz continuous with $\|\mathcal{S}_{\mu}\|_{\textnormal{Lip}}\leq C$,
causal and independent of $\mu$, where $\mathcal{M},\mathcal{N}$
denote the abstract multiplication operators given by $(\mathcal{M}\phi)(t)=M(t)\phi(t)$
and $(\mathcal{N}\phi)(t)=N(t)\phi(t)$, respectively. In particular,
$\mathcal{S}_{\mu}|_{\mathring{C}_{\infty}(\mathbb{R};H)}$ is densely defined,
causal and evolutionary at $\nu$. \end{theorem}

Although the latter theorem is a direct analogue of Theorem \ref{thm:st}
in the nonlinear setting, its proof is completely different and rests
on perturbation results for maximal monotone relations. As the proof
is quite long and technical, we omit it here and refer to \cite{Trostorff2014} instead.

\section{Stochastic evolutionary equations}

\label{sec:stochastics}

Similar to the approach outlined in \cite{SW16_SD}, we present the
solution theory for stochastic partial differential equations based
on Theorem \ref{thm:cmp}. For this we first need to establish a suitable
functional analytic formulation for the stochastic integral. In contrast
to \cite{SW16_SD}, where the authors focused on the case of Hilbert space
valued Wiener processes, we shall favor a more axiomatic approach here.
Indeed, this gives us more freedom for the choice of the stochastic processes
in the integral. For this, we will introduce a class of `admissible' processes
and corresponding stochastic integrals. We mainly follow the rationale
presented in \cite{Metivier1980}.

\subsection{An abstract description of stochastic integration}

\label{sec:SI}

Throughout, we denote by $(\Omega,\Sigma,\mathbb{P})$ a probability
space. Moreover, we fix a filtration $\mathcal{F}\coloneqq(\Sigma_{t})_{t\in\mathbb{R}}$,
i.e. a family of sub-$\sigma$-algebras of $\Sigma$ satisfying 
\[
\Sigma_{s}\subseteq\Sigma_{t}\quad(s\leq t).
\]
Moreover, we fix separable Hilbert spaces $G,H$ and a subspace $L\subseteq L(G,H)$
equipped with a Banach norm $\|\cdot\|_{L}$ such that 
\[
(L,\|\cdot\|_{L})\hookrightarrow(L(G,H),\|\cdot\|).
\]
\begin{definition} We collect some notions, which are needed in the following. 
\begin{enumerate}[label=(\alph{enumi})]
\item We consider the following collection of sets 
\[
\big\{]s,t]\times A\,;\,s,t\in\R,s<t,A\in\Sigma_{s}\big\}\subseteq\mathcal{P}(\R\times\Omega).
\]
The $\sigma$-algebra generated by those sets is denoted by $\mathcal{B}_{\mathcal{F}}$
and is called the $\sigma$-algebra of \emph{$\mathcal{F}$-predictable
sets}. 
\item A mapping $X:\R\times\Omega\to Z$, where $Z$ is a Banach space,
is called a \emph{stochastic process}, if for each $t\in\R$ the mapping
$X_{t}=X(t,\cdot):\Omega\to Z$ is measurable. 
\item A stochastic process $X:\R\times\Omega\to Z$ is called \emph{$\mathcal{F}$-adapted},
if $X_{t}$ is $\Sigma_{t}$-measurable for each $t\in\R$. $X$ is
called \emph{$\mathcal{F}$-predictable}, if $X$ is $\mathcal{B}_{\mathcal{F}}$-measurable.
For $\nu\geq0$ we define 
\[
L_{\nu,\mathrm{pr}}^{2}\big(\R;L^{2}(\P;Z)\big)\coloneqq\{X\in L_{\nu}^{2}\big(\R;L^{2}(\P;Z)\big)\,;\,X\text{ predictable}\},
\]
which is a closed subspace of $L_{\nu}^{2}\big(\R;L^{2}(\P;Z)\big)$. 
\end{enumerate}
\end{definition}

We now fix a stochastic process $X$ attaining values in $G$. The
goal is now to define stochastic integration with respect to this
process $X$. The integrands are suitable stochastic processes attaining
values in $L$ and the integral should be an element in $H^{\Omega}$.

We start by defining 
\[
I^{X}(T\chi_{]s,t]\times A})\coloneqq\intop_{\R}T\chi_{]s,t]\times A}\,\d X\coloneqq(\omega\mapsto\chi_{A}(\omega)T(X_{t}(\omega)-X_{s}(\omega)).
\]
where $s<t,A\in\Sigma_{s}$ and $T\in L$. Clearly, this integral
operator $I^{X}$ can be extend to a linear operator on simple $\mathcal{F}$-predictable
processes $Y:\mathbb{R}\times\Omega\to L$. We denote this linear
extension again by $I^{X}$.

Moreover, if $X_{t}\in L^{2}(\mathbb{P};G)$ for each $t\in\R$ we
immediately get that $I^{X}$ attains values in $L^{2}(\mathbb{P};H)$.
The main idea is now to extend this integral operator to a broader
class of processes. For doing so, we need to restrict to a certain
class of processes $X$.

\begin{definition} \label{def:pr} Let $X:\mathbb{R}\times\Omega\to G$ be such that $X_{t}\in L^{2}(\mathbb{P};G)$
for each $t\in\R$. We call $X$ an \emph{$L^{2}$-primitive}, if
there exists a measure $\alpha:\mathcal{B}_{\mathcal{F}}\to[0,\infty]$
and $C\geq0$ such that 
\[
\Big|I^{X}\big(\sum_{i=0}^{n}T_{i}\chi_{]s_{i},t_{i}]\times A_{i}}\big)\Big|_{L^{2}(\mathbb{P};H)}\leq C\Big|\sum_{i=0}^{n}T_{i}\chi_{]s_{i},t_{i}]\times A_{i}}\Big|_{L^{2}(\alpha;L)},
\]
where $n\in\N,s_{i}<t_{i},A_{i}\in\Sigma_{s_{i}}$ and $T_{i}\in L$, $i\in \{0,\ldots,n\}$.
\\
 In this case, $\alpha$ is called a \emph{dominating measure for
$X$}. We denote by $I^{X,\alpha}$ the unique extension of $I^{X}$
to a bounded linear operator 
\[
I^{X,\alpha}:L^{2}(\alpha;L)\to L^{2}(\mathbb{P};H)
\]
and call it \emph{the stochastic integral with respect to $X$} on
$L^{2}(\alpha;L)$. For $Y\in L^{2}(\alpha;L)$ we also write 
\[
\intop_{\R}Y\,\d X\coloneqq I^{X,\alpha}(Y).
\]

\end{definition}

\begin{remark} We note that in the latter definition the measure $\alpha$
may not be uniquely determined. Thus, the latter definition allows
for the extension of the stochastic integral in various spaces. It
is clear, however, that for two dominating measures $\alpha_{1},\alpha_{2}$
we have that the two extension $I^{X,\alpha_{1}}$ and $I^{X,\alpha_{2}}$
coincide on the intersection $L^{2}(\alpha_{1};L)\cap L^{2}(\alpha_{2};L)$.
\end{remark}

\begin{ex}\label{ex:sta}By \cite[Section 2.6]{Metivier1980} the following processes
are $L^{2}$-primitives. 
\begin{enumerate}[label=(\alph{enumi}), ref=(\alph{enumi})]
\item \label{independent} Let $X$ satisfy $X_{t}\in L^{2}(\Omega,\Sigma_{t},\mathbb{P};G)$
for each $t\in\mathbb{R}$ (i.e., $X$ is $\mathcal{F}$-\emph{adapted})
and assume that 
\[
\R\ni t\mapsto|X_{t}|_{L^{2}(\P;G)}
\]
is right continuous. Moreover, we assume that $X$ has independent
and centered increments, i.e., for each $s,t\in\R$ with $s<t$ and
each $x\in G$ we have that 
\[
\omega\mapsto\langle X_{t}(\omega)-X_{s}(\omega),x\rangle_{G}
\]
is independent of $\Sigma_{s}$ and 
\[
\intop_{\Omega}\big(X_{t}(\omega)-X_{s}(\omega)\big)\,\d\mathbb{P}(\omega)=0.
\]
Then $X$ is an $L^{2}$-primitive with dominating measure $\alpha=\mu\otimes\P$,
where $\mu$ is the Stieltjes measure associated with the function
\[
t\mapsto|X_{t}|_{L^{2}(\P;G)}.
\]
In particular, the Hilbert space valued Wiener process $W$ is an
$L^{2}$-primitive, if we choose $\Sigma_{t}\coloneqq\sigma(W_{s};s\leq t)$. 
\item More generally, if $X:\mathbb{R}\times\Omega\to G$ is an $\mathcal{F}$-martingale,
such that $X_{t}\in L^{2}(\Omega,\Sigma_{t},\P;G)$ for each $t\in\R$
and 
\[
t\mapsto|X_{t}|_{L^{2}(\mathbb{P};G)}
\]
is right-continuous, then $X$ is an $L^{2}$-primitive with dominating
measure $\alpha=d_{|X|_{G}^{2}}$, the \emph{Doleans-measure} of the
submartingale $|X|_{G}^{2}$ (see e.g. \cite[Section 1.20]{Metivier1980}
or \cite{Doleans1968}). 
\end{enumerate}
\end{ex}

Our next goal is to introduce a primitive of an $L$-valued process
$Y$ with respect to a $G$-valued process $X$.

\begin{prop}\label{prop:int_evo} Let $X:\R\times\Omega\to G$ be
an $L^{2}$-primitive with dominating measure $\alpha$. Moreover,
for $\nu>0$ we define the space $L_{\nu}^{2}(\alpha;L)$ as the space
of $L$-valued, $\mathcal{F}$-predictable processes $Y$ satisfying
\[
\intop_{\R\times\Omega}|Y(t,\omega)|_{L}^{2}\exp(-2\nu t)\,\d\alpha(t,\omega)<\infty.
\]
For $\mu>0$ we consider the operator 
\[
\mathcal{I}^{X,\alpha}:S(\alpha;L)\subseteq\bigcap_{\nu\geq\mu}L_{\nu}^{2}(\alpha;L)\to\bigcap_{\nu\geq\mu}L_{\nu}^{2}\big(\R;L^{2}(\P;H)\big)
\]
 given by 
\[
\mathcal{I}^{X,\alpha}(Y)\coloneqq(t\mapsto I^{X}(\chi_{\R_{\leq t}}Y)),
\]
where $S(\alpha;L)$ denotes the space of simple $L$-valued, predictable
processes.

Then $\mathcal{I}^{X,\alpha}$ is evolutionary at $\mu$ and densely
defined. More precisely, there exists a constant $C\geq0$ such that
\[
\|\mathcal{I}_{\nu}^{X,\alpha}\|\leq\frac{C}{\sqrt{2\nu}}\quad(\nu\geq\mu).
\]
Moreover, $\mathcal{I}_{\nu}^{X,\alpha}$ is causal and $\mathcal{I}_{\nu}^{X,\alpha}$
and $\mathcal{I}_{\tilde{\nu}}^{X,\alpha}$ coincide on the intersection
$L_{\nu}^{2}(\alpha;L)\cap L_{\tilde{\nu}}^{2}(\alpha;L)$ for each
$\nu,\tilde{\nu}\geq\mu$. \end{prop}

\begin{proof} Let $Y\in S(\alpha;L)$. First we note that $\chi_{\R_{\leq t}}Y\in S(\alpha;L)$
for each $t\in\R$. Let now $\nu\geq\mu$. We estimate 
\begin{align*}
& \intop_{\R}|\mathcal{I}^{X,\alpha}(Y)(t)|_{L^{2}(\P;H)}^{2}\exp(-2\nu t)\,\d t \\ & \leq C^{2}\intop_{\R}\intop_{\R\times\Omega}|\chi_{\R_{\leq t}}(s)Y(s,\omega)|_{L}^{2}\,\d\alpha(s,\omega)\exp(-2\nu t)\,\d t\\
 & =C^{2}\intop_{\R\times\Omega}\intop_{\R}\chi_{\R_{\geq s}}(t)\exp(-2\nu t)\,\d t\,|Y(s,\omega)|_{L}^{2}\,\d\alpha(s,\omega)\\
 & =\frac{C^{2}}{2\nu}|Y|_{L_{\nu}^{2}(\alpha;L)}^{2},
\end{align*}
which shows that $\mathcal{I}^{X,\alpha}$ is evolutionary at $\mu$
and that the norm estimate holds. The causality and the independence
on the parameter $\nu$ follows from Lemma \ref{lem:criteriona_for_causal}
and Lemma \ref{lem:indep_of_evo}. \end{proof}

\begin{lem}\label{lem:int_predictable} Let $X:\R\times\Omega\to G$
be an $L^{2}$-primitive with dominating measure $\alpha$ and assume
that 
\[
\R\ni t\mapsto X_{t}(\omega)
\]
is weakly left continuous. Moreover, we assume that $X$ is $\mathcal{F}$-adapted.
Then for $\nu>0$ and $Y\in L_{\nu}^{2}(\alpha;L)$ we have that 
\[
\mathcal{I}^{X,\alpha}(Y)\in L_{\nu,\mathrm{pr}}^{2}(\R;L^{2}(\P;H)).
\]
Consequently, 
\[
\mathcal{I}^{X,\alpha}:L_{\nu}^{2}(\alpha;L)\to L_{\nu,\mathrm{pr}}^{2}(\R;L^{2}(\P;H))
\]
is a bounded linear operator. \end{lem}

\begin{proof} It suffices to prove that $\mathcal{I}^{X,\alpha}(Y)$
is predictable. Due to linearity and continuity it suffices to consider
the case $Y=\chi_{]s,t]\times A}T$ for some $s<t,A\in\Sigma_{s},T\in L$.
Then we have 
\[
\mathcal{I}^{X,\alpha}(Y)(\tau,\omega)=\chi_{]s,t]\times A}(\tau,\omega)T(X_{\tau}(\omega)-X_{s}(\omega))+\chi_{]t,\infty]\times A}(\tau,\omega)T(X_{t}(\omega)-X_{s}(\omega))
\]
for each $\tau\in\R,\omega\in\Omega$. Note that $\mathcal{I}^{X,\alpha}(Y)$
is $\mathcal{F}$-adapted and that $\tau\mapsto\mathcal{I}^{X,\alpha}(Y)(\tau,\omega)$
is weakly left continuous for each $\omega\in\Omega$. Thus, by \cite[Proposition 3.7]{dapratozabczyk}
it is $\mathcal{F}$-predictable. \end{proof}

\subsection{Solution theory for abstract stochastic evolutionary equations}\label{sec:stsee}

In the previous section, we have focused on the stochastic part of
the evolutionary equation with stochastic perturbation. We are now in the position to combine the results of the previous sections in order to provide the desired solution theory. First of all, we state the
main assumptions of this section.

\begin{ass}\label{ass:stochastics} Let $G,H$ be two separable Hilbert
spaces and $L\subseteq L(G,H)$ a subspace equipped with a Banach
norm, such that $L\hookrightarrow L(G,H)$. Moreover, let $(\Omega,\Sigma,\mathbb{P})$
be a probability space and $\mathcal{F}=(\Sigma_{t})_{t\in\R}$ a
filtration. We fix an $\mathcal{F}$-adapted process $X:\R\times\Omega\to G$,
which is an $L^{2}$-primitive with dominating measure $\alpha$ and
we assume that 
\[
t\mapsto X_{t}(\omega)
\]
is weakly left continuous for each $\omega\in\Omega$. Moreover, we
assume that $L_{0,\mathrm{pr}}^{2}(\R;L^{2}(\P))\hookrightarrow L^{2}(\alpha)$.
\end{ass}

We first recall the central observation of \cite{SW16_SD} in a slightly
different way.

\begin{theorem}[{{\cite[Theorem 3.4]{SW16_SD}}}]\label{thm:causal,adap}
Let $\nu>0$ and $M\in L(L_{\nu}^{2}(\R;H))$ be causal. Then the
canonical extension of $M$ to $L_{\nu}^{2}(\R;L^{2}(\P;H))$ given
by 
\[
(Mu)(t,\omega)\coloneqq M(u(\cdot,\omega))(t)\quad(t\in\R,\omega\in\Omega)
\]
leaves the space of predictable processes invariant, that is,
\[
M[L_{\nu,\mathrm{pr}}^{2}(\R;L^{2}(\P;H))]\subseteq L_{\nu,\mathrm{pr}}^{2}(\R;L^{2}(\P;H)).
\]
\end{theorem}

Before we can come to our main well-posedness result, we need the
following lemma.

\begin{lem}\label{lem:sigma} Let $\mu>0$, $\sigma:H\to L$ be Lipschitz continuous, $\sigma(0)=0$,
and define 
\[
\tilde{\sigma}:\bigcap_{\nu\geq\mu}L_{\nu}^{2}(\R;L^{2}(\P;H))\to\bigcap_{\nu\geq\mu}L_{\nu}^{2}(\R;L^{2}(\P;L))
\]
by 
\[
(\tilde{\sigma}u)(t,\omega)\coloneqq\sigma(u(t,\omega)).
\]
Then $\tilde{\sigma}$ is evolutionary at $\mu$, $\tilde{\sigma}_{\nu}$
is causal and does not depend on the parameter $\nu$. Moreover, for
$\nu\geq\mu$ we have that the restriction 
\[
\tilde{\sigma}_{\nu}:L_{\nu,\mathrm{pr}}^{2}(\R;L^{2}(\P;H))\to L_{\nu}^{2}(\alpha;L)
\]
is well-defined and Lipschitz continuous, where the smallest Lipschitz constant
can be chosen independent of $\nu$. \end{lem}

\begin{proof} It is obvious, that $\tilde{\sigma}$ is again Lipschitz continuous
with the same Lipschitz constant as $\sigma$. Hence, it is evolutionary
at $\mu$. The causality and independence of the parameter follows
by Lemma \ref{lem:indep_of_evo} and Lemma \ref{lem:criteriona_for_causal}.
We now come to the second assertion. We first show that $\tilde{\sigma}(u)$
is predictable if $u$ is predictable. Note that by continuity it
suffices to prove this for $u=\sum_{i=1}^{n}\chi_{]s_{i},t_{i}]\times A_{i}}x_{i}$
for $s_{i}<t_{i},A_{i}\in\Sigma_{s_{i}},x_{i}\in H$, $i\in\{1,\ldots,n\}$. We may assume
without loss of generality that the intervals $]s_{i},t_{i}]$ are
pairwise disjoint. Then 
\[
\tilde{\sigma}(u)=\sum_{i=1}^{n}\chi_{]s_{i},t_{i}]\times A_{i}}\sigma(x_{i})
\]
and hence, $\tilde{\sigma}(u)$ is predictable. From $L_{0,\mathrm{pr}}^{2}(\R;L^{2}(\P))\hookrightarrow L^{2}(\alpha)$
we infer \[L_{\nu,\mathrm{pr}}^{2}\big(\R;L^{2}(\P;L)\big)\hookrightarrow L_{\nu}^{2}(\alpha;L)\]
for each $\nu\geq0$ where the embedding constant is independent of
$\nu$. Thus, as we have shown that 
\[
\tilde{\sigma}_{\nu}:L_{\nu,\mathrm{pr}}^{2}(\R;L^{2}(\P;H))\to L_{\nu,\mathrm{pr}}^{2}(\R;L^{2}(\P;L))
\]
is well-defined and Lipschitz continuous with a Lipschitz constant
independent of $\nu$, the assertion follows. \end{proof}

\begin{cor}\label{cor:inv_evo} Let $\mu>0$, $\sigma:H\to L$ be
Lipschitz continuous, $\sigma(0)=0$. Then 
\[
\mathcal{I}^{X,\alpha}\circ\tilde{\sigma}:\bigcap_{\nu\geq\mu}L_{\nu,\mathrm{pr}}^{2}(\R;L^{2}(\P;H))\subseteq\bigcap_{\nu\geq\mu}L_{\nu}^{2}(\R;L^{2}(\P;H))\to\bigcap_{\nu\geq\mu}L_{\nu}^{2}(\R;L^{2}(\P;H))
\]
is invariant evolutionary and $\|\mathcal{I}^{X,\alpha}\circ\tilde{\sigma}\|_{\mathrm{ev,Lip}}=0$.
\end{cor}

\begin{proof} By Proposition \ref{prop:int_evo} and Lemma \ref{lem:sigma}
we infer that $\mathcal{I}^{X,\alpha}\circ\tilde{\sigma}$ is evolutionary
with $\|\mathcal{I}^{X,\alpha}\circ\tilde{\sigma}\|_{\mathrm{ev,Lip}}=0$.
The invariance follows by Lemma \ref{lem:int_predictable} and Lemma
\ref{lem:sigma}. \end{proof}

We are now in the position to formulate the abstract solution theory,
which is based on Theorem \ref{thm:cmp}.

\begin{theorem}\label{thm:linst} Impose Assumption \ref{ass:1} and
let $\sigma\colon H\to L$ be Lipschitz continuous, $\sigma(0)=0$. Then there exists
$\nu>0$ such that for all $f\in L_{\nu,\mathrm{pr}}^{2}(\mathbb{R};L^{2}(\P;H))$
there exists a unique $u_{f}\in L_{\nu,\mathrm{pr}}^{2}(\mathbb{R};L^{2}(\P;H))$
such that 
\[
\overline{\left(\partial_{0,\nu}\mathcal{M}+\mathcal{N}+A\right)}u_{f}=f+\mathcal{I}_{\nu}^{X,\alpha}\circ\tilde{\sigma}(u_{f}).
\]
The mapping $f\mapsto u_{f}$ is causal and does not depend on $\nu$
in the sense of Lemma \ref{lem:indep_of_evo}. \end{theorem} \begin{proof}
We use Theorem \ref{thm:cmp}. For this, we observe that $S\coloneqq\left(\partial_{0,\nu}\mathcal{M}+\mathcal{N}+A\right)^{-1}$
is evolutionary, causal and densely defined, by Theorem \ref{thm:st}.
Moreover, since $S$ is causal, we have that $S_{\mu}[\dom(F_{\mu})]\subseteq\dom(F_{\mu})$
for all $\mu\geq\nu$ with $F\coloneqq\mathcal{I}^{X,\alpha}\circ\tilde{\sigma}$,
by Theorem \ref{thm:causal,adap}. Moreover, by Corollary \ref{cor:inv_evo},
we deduce that $F$ is invariant evolutionary and $\|F\|_{\textnormal{ev},\textnormal{Lip}}=0$.
Thus Theorem \ref{thm:cmp} is applicable and we obtain the assertion.
\end{proof}

Also in the non-linear setting, we obtain an analogous result with
exactly the same proof, where we use Theorem \ref{thm:nonlinear}
instead of Theorem \ref{thm:st}.

\begin{theorem}\label{thm:wp_stoch_nonlin} Let $M,N:\mathbb{R}\to L(H)$
weakly measurable and $M$ Lipschitz continuous. Let $A\subseteq H\otimes H$
maximal monotone with $(0,0)\in A$ and $\sigma:H\to L$ Lipschitz
continuous, $\sigma(0)=0$. Moreover, assume that $K\coloneqq\kar(M(t))=\kar(M(0))$
for all $t\in\mathbb{R}$ and that there exists $c>0$ such that for
all $t\in\mathbb{R}$ 
\[
\langle M(t)\phi,\phi\rangle\geq c\langle\phi,\phi\rangle,\ \Re\langle N(t)\psi,\psi\rangle\geq c\langle\psi,\psi\rangle
\]
for all $\phi\in K$ and $\psi\in K^{\bot}$. Then there exists $\nu>0$
such that for all $f\in L_{\nu,\mathrm{pr}}^{2}(\mathbb{R};L^{2}(\P;H))$
there exists a unique $u_{f}\in L_{\nu,\mathrm{pr}}^{2}(\mathbb{R};L^{2}(\P;H))$
such that 
\[
\overline{\left(\partial_{0,\nu}\mathcal{M}+\mathcal{N}+A\right)}\ni(u_{f},f+\mathcal{I}_{\nu}^{X,\alpha}\circ\tilde{\sigma}(u_{f})).
\]
The mapping $f\mapsto u_{f}$ is causal and does not depend on $\nu$
in the sense of Lemma \ref{lem:indep_of_evo}. \end{theorem}

\subsection{An abstract stochastic heat/wave equation}\label{sec:wave}

In this section, we treat an abstract example of an equation of mixed
type. For this let $H_{1},H_{2}$ be separable Hilbert spaces, $C\colon\dom(C)\subseteq H_{1}\to H_{2}$
be closed and densely defined. We assume Assumption \ref{ass:stochastics} with
$H$ replaced by $H_{1}$. Let $a\colon\mathbb{R}\to L(H_{2})$ be
bounded and Lipschitz continuous satisfying $a(t)=a(t)^{*}\geq c$
for all $t\in\mathbb{R}$ and some $c>0$; we denote by $a(\mathrm{m})$
the abstract multiplication operator realized as an operator from
$L_{\nu}^{2}(\mathbb{R};H_{2})$ to $L_{\nu}^{2}(\mathbb{R};H_{2})$
for all $\nu>0$. Moreover, let $\sigma:H_{1}\to L$ be Lipschitz continuous with $\sigma(0)=0$
and $P=P^{\ast}=P^{2}\in L(H_{1})$.

The problem we are about to study with regards to well-posedness issues
reads as follows. Let $f\in L_{\nu,\mathrm{pr}}^{2}(\R;L^{2}(\P;H_{1}))$
be given. Then (write $\partial_0$ for the time derivatve) consider the equation 
\begin{equation}
\partial_{0}^{2}Pu+\partial_{0}(1-P)u+C^{*}a(\mathrm{m})Cu=f+\partial_{0}\left(\mathcal{I}^{X,\alpha}\circ\tilde{\sigma}\right)(u).\label{eq:wave0}
\end{equation}
Note that for the special cases $P=1_{H_{1}}$ and $P=0$, we recover
the respective special cases of an abstract wave equation and an abstract
heat equation.

In applications, for instance if $X$ is a Wiener process, the expression
$\partial_{0}\mathcal{I}^{X,\alpha}\circ\tilde{\sigma}(u)$ is often
written as $\sigma(u(t))\dot{W}(t)$. Then $\int_{0}^{t}\sigma(u(t))\dot{W}(t)dt$
is interpreted as stochastic integral. Here, we employ the same rationale
since $\partial_{0}^{-1}$ is integration (see Section \ref{sec:tdem})
so that $\partial_{0}^{-1}\partial_{0}\mathcal{I}^{X,\alpha}\circ\tilde{\sigma}=\mathcal{I}^{X,\alpha}\circ\tilde{\sigma}$.

In order to apply the solution theory outlined in the previous section,
we shall reformulate \eqref{eq:wave0}. Denote $q\coloneqq\partial_{0}^{-1}a(\mathrm{m})Cu$.
Thus, \eqref{eq:wave0} reads 
\begin{multline}
\left(\partial_{0}\begin{pmatrix}P & 0\\
0 & b(\mathrm{m})
\end{pmatrix}+\begin{pmatrix}(1-P) & 0\\
0 & -b'(\mathrm{m})
\end{pmatrix}+\begin{pmatrix}0 & C^{*}\\
-C & 0
\end{pmatrix}\right)\begin{pmatrix}u\\
q
\end{pmatrix} \\ =\begin{pmatrix}\partial_{0}^{-1}f+\mathcal{I}^{X,\alpha}\circ\tilde{\sigma}(u)\\
0
\end{pmatrix},\label{eq:wave1sto}
\end{multline}
where $b(t)\coloneqq a(t)^{-1}$ for all $t\in\mathbb{R}$ and $b'$
is the weak derivative of $b$.

Identifying $x\in H_{1}$ with $x\oplus0\in H_{1}\oplus H_{2}$ we
realize that Theorem \ref{thm:linst} applies once we have shown that
Assumption \ref{ass:1} is satisfied for the following setting 
\begin{align*}
 & H=H_{1}\oplus H_{2},\mathcal{M}=\begin{pmatrix}P & 0\\
0 & b(\mathrm{m})
\end{pmatrix},\mathcal{N}=\begin{pmatrix}(1-P) & 0\\
0 & -b'(\mathrm{m})
\end{pmatrix},\\ & A=\begin{pmatrix}0 & C^{*}\\
-C & 0
\end{pmatrix}, \mathcal{M}'=\begin{pmatrix}0 & 0\\
0 & b'(\mathrm{m})
\end{pmatrix},
\end{align*}
and a suitably chosen $\nu>0$. The rest of this section is devoted
to verify the conditions in Assumption \ref{ass:1}.

\begin{lem}\label{lem:Asks} The operator $A$ is skew-selfadjoint.
In particular, $\Re\langle Ax,x\rangle=0$ for all $x\in\dom(A)=\dom(A^{*})$
and $\ran(A\pm1)=H$, so that $A$ is m-accretive. \end{lem} \begin{proof}
The claim follows once we realize that for densely defined operators
$B_{1},B_{2}$ acting in appropriate Hilbert space, we have $\begin{pmatrix}0 & B_{1}^{*}\\
B_{2}^{*} & 0
\end{pmatrix}=\begin{pmatrix}0 & B_{2}\\
B_{1} & 0
\end{pmatrix}^{*}$.\end{proof}

\begin{lem}\label{lem:Mcom} For all $\nu\in\mathbb{R}$, we have
\[
\partial_{0,\nu}\mathcal{M}\subseteq\mathcal{M}\partial_{0,\nu}-\mathcal{M}'.
\]
\end{lem} \begin{proof} Let $s,t\in\mathbb{R}$. Then we compute
\begin{align*}
\|b(t)-b(s)\| & =\|a(t)^{-1}-a(s)^{-1}\|\\
 & =\|a(t)^{-1}\left(a(s)-a(t)\right)a(s)^{-1}\|\\
 & \leq\frac{1}{c^{2}}\|a(s)-a(t)\|\leq\frac{1}{c^{2}}\|a\|_{\textnormal{Lip}}|s-t|.
\end{align*}
Thus, $b$ is Lipschitz continuous and $H_{2}$ is separable, thus,
$b\phi$ is weakly differentiable for all $\phi\in\mathring{C}_{\infty}(\mathbb{R};H_{2})$
and 
\[
(b\phi)'=b'\phi+b\phi',
\]
where $b'$ is the strong derivative defined by the derivative of
$b'(t)x\coloneqq(b(\cdot)x)'(t)$ for almost all $t\in\mathbb{R}$,
$x\in H_{2}$, see also \cite[Lemma 2.1]{PTWW13_NA}. Since $P\partial_{0,\nu}\subseteq\partial_{0,\nu}P$,
the claim follows. \end{proof}

A particular observation in the latter proof is that $b$ is Lipschitz continuous
and $\|b\|_{\textnormal{Lip}}\leq\frac{1}{c^{2}}\|a\|_{\textnormal{Lip}}$.

\begin{lem}\label{lem:pdwh} There exist $k,\nu>0$ such that for
all $t\in\mathbb{R}$ and $\phi\in\mathring{C}_{\infty}(\mathbb{R};H)$
we have 
\[
\Re\langle\left(\partial_{0,\nu}\mathcal{M}+\mathcal{N}\right)\phi,Q_{t}\phi\rangle_{L_{\nu}^{2}}\geq k\langle Q_{t}\phi,\phi\rangle_{L_{\nu}^{2}}.
\]
\end{lem} \begin{proof} Since $\mathcal{M},\mathcal{M}',\mathcal{N}$
are diagonal block operator matrices, it suffices to restrict ourselves
to test functions $\phi$ in $\mathring{C}_{\infty}(\mathbb{R};H_{1})$
and $\mathring{C}_{\infty}(\mathbb{R};H_{2})$. We can apply \cite[Lemma 2.6]{PTWW13_NA}
to $M_{0}(\mathrm{m})=(1-P)$, $M_{1}(\mathrm{m})=P$, $A=0$ to obtain
for all $t\in\mathbb{R}$, $\nu>0$: 
\[
\Re\langle\partial_{0,\nu}P\phi+(1-P)\phi,Q_{t}\phi\rangle_{L_{\nu}^{2}}=\frac{1}{2}\langle\phi(t),P\phi(t)\rangle_{H_{1}}e^{-2\nu t}+\langle\nu P\phi+(1-P)\phi,Q_{t}\phi\rangle_{L_{\nu}^{2}}.
\]
Thus, for all $\phi\in\mathring{C}_{\infty}(\mathbb{R};H_{1})$ 
\[
\Re\langle\partial_{0,\nu}P\phi+(1-P)\phi,Q_{t}\phi\rangle_{L_{\nu}^{2}}\geq\min\{\nu,1\}\langle\phi,Q_{t}\phi\rangle_{L_{\nu}^{2}}.
\]
Next, again by \cite[Lemma 2.6]{PTWW13_NA} this time applied to $M_{0}(\mathrm{m})=b(\mathrm{m})$,
$M_{1}(\mathrm{m})=-b'(\mathrm{m})$, and $A=0$, we compute 
\[
\Re\langle\partial_{0,\nu}b(\mathrm{m})\phi-b'(\mathrm{m})\phi,Q_{t}\phi\rangle_{L_{\nu}^{2}}  =\frac{1}{2}\langle\phi(t),b(t)\phi(t)\rangle_{H_{1}}e^{-2\nu t}+\langle\nu b(\mathrm{m})\phi-\frac{1}{2}b'(\mathrm{m})\phi,Q_{t}\phi\rangle_{L_{\nu}^{2}},
\]
where we used that $b'(t)$ is selfadjoint as $b(s)$ is selfadjoint
for all $s\in\mathbb{R}$. We observe that the non-negativity of $a$
implies the same for $b$; more precisely we get $b(t)\geq\frac{c}{\|a(t)\|^{2}}\geq\frac{c}{\sup_{t\in\mathbb{R}}\|a(t)\|^{2}}\eqqcolon c'>0$.
Moreover, we note that $\|b'(\mathrm{m})\|\leq\|b\|_{\textnormal{Lip}}$,
by \cite[Lemma 2.1]{PTWW13_NA}. Hence, $\|b'(\mathrm{m})\|\leq\frac{1}{c^{2}}\|a\|_{\textnormal{Lip}}$.
Thus, we deduce that 
\begin{align*}
\Re\langle\partial_{0,\nu}b(\mathrm{m})\phi-b'(\mathrm{m})\phi,Q_{t}\phi\rangle_{L_{\nu}^{2}} & \geq\langle\nu b(\mathrm{m})\phi-\frac{1}{2}b'(\mathrm{m})\phi,Q_{t}\phi\rangle_{L_{\nu}^{2}}\\
 & =\langle(\nu b(\mathrm{m})-\frac{1}{2}b'(\mathrm{m}))Q_{t}\phi,Q_{t}\phi\rangle_{L_{\nu}^{2}}\\
 & \geq\left(\nu c'-\frac{1}{c^{2}}\|a\|_{\textnormal{Lip}}\right)\langle Q_{t}\phi,Q_{t}\phi\rangle_{L_{\nu}^{2}},
\end{align*}
which yields the assertion. \end{proof}

The Lemmas \ref{lem:Asks}, \ref{lem:Mcom}, and \ref{lem:pdwh} finally
yield the applicability of Theorem \ref{thm:linst}, so that \eqref{eq:wave1sto}
is well-posed. In fact, we have the following result:

\begin{theorem}\label{thm:wh} There exists $\nu>0$ such that for all $f\in L_{\nu,\textnormal{pr}}^2(\mathbb{R};L^2(\mathbb{P};H_1))$ we find a uniquely determined $u_f \in L_{\nu,\textnormal{pr}}^2(\mathbb{R};L^2(\mathbb{P};H_1))$ and $q_f \in L_{\nu,\textnormal{pr}}^2(\mathbb{R};L^2(\mathbb{P};H_2))$ such that
\[
  \left(\overline{ \partial_{0,\nu}\begin{pmatrix} P & 0 \\ 0 & b(\mathrm{m})
   \end{pmatrix} + \begin{pmatrix} 1-P & 0 \\ 0 & -b'(\mathrm{m}) \end{pmatrix} + \begin{pmatrix} 0 & \dive \\ \interior{\grad} & 0 \end{pmatrix} } \right)  
   \begin{pmatrix} u_f \\ q_f \end{pmatrix}  = \begin{pmatrix} \partial_{0,\nu}^{-1} f + \mathcal{I}^{X,\alpha}\circ \tilde{\sigma}(u_f) \\ 0\end{pmatrix}.
\]
If $f=0$ on $(-\infty,t]$ for some $t\in \mathbb{R}$, then $u_f=0$ and $q_f=0$ on $(-\infty,t]$.
\end{theorem}

\begin{remark}\label{rem:distr} Using the notion of extrapolation spaces,
we can make sense of the expression $\sigma(u(t))\dot{X}(t)$, which
might seem to be quite formal at first glance. For this we define
the spaces $L_{\nu,\mathrm{pr},-1}^{2}(\R;L^{2}(\P;H))$ as the completion
of $L_{\nu,\mathrm{pr}}^{2}(\R;L^{2}(\P;H))$ with respect to the
norm ${\left|\partial_{0,\nu}^{-1}\cdot\right|}_{L_{\nu}^{2}(\R;L^{2}(\P;H))}$ and
$H_{-1}(A)$ as the completion of $H$ with respect to $|(1+A)^{-1}\cdot|_{H}$.
Now, fix $f$ and let $(u,q)$ solve \eqref{eq:wave1sto}. Then we
have that 
\[
\left(\overline{\partial_{0,\nu}\begin{pmatrix}P & 0\\
0 & b(\mathrm{m})
\end{pmatrix}+\begin{pmatrix}(1-P) & 0\\
0 & -b'(\mathrm{m})
\end{pmatrix}+\begin{pmatrix}0 & C^{*}\\
-C & 0
\end{pmatrix}}\right)\begin{pmatrix}u\\
q
\end{pmatrix} =\begin{pmatrix}\partial_{0}^{-1}f+\mathcal{I}^{X,\alpha}\circ\tilde{\sigma}(u)\\
0
\end{pmatrix}
\]
as an equation in $L_{\nu,\mathrm{pr}}^{2}(\mathbb{R};L^{2}(\P;H))$
for some large enough $\nu>0$. Using the continuous extensions of
$\partial_{0,\nu}$ and $A=\begin{pmatrix}0 & C^{*}\\
-C & 0
\end{pmatrix}$ with values in $L_{\nu,\mathrm{pr},-1}^{2}\big(\mathbb{R};L^{2}(\P;H)\big)$
and $L_{\nu,\mathrm{pr}}^{2}\big(\mathbb{R};L^{2}(\P;H_{-1}(A))\big)$, respectively,
we deduce that the equation satisfied by $(u,q)$ can be written without
the closure bar. Moreover, one can differentiate line by line to obtain
\[
\partial_{0,\nu}^{2}Pu+\partial_{0,\nu}(1-P)u+C^{*}\partial_{0,\nu}q=f+\partial_{0}\mathcal{I}^{X,\alpha}\circ\tilde{\sigma}(u)
\]
and 
\[
a(\mathrm{m})^{-1}\partial_{0,\nu}q=Cu,
\]
which in turn formally gives back the system one started out with.
\end{remark}

We remark that the well-posedness of \eqref{eq:wave1sto} is also
covered by Theorem \ref{thm:wp_stoch_nonlin}, as $\mathcal{M}$ and
$\mathcal{N}$ are given as abstract multiplication operators. Hence,
we could generalize \eqref{eq:wave1sto} by replacing the operator
$A$ by a maximal monotone relation. This allows for instance the
treatment of certain hysteresis effects in the theory of plasticity
(see \cite{Trostorff2012,Trostorff2014}) or Section \ref{sec:nex} above.

\begin{remark} We shall comment on the limitations of the approach at hand.

(a) First of all, the developed theory is a Hilbert space approach and so certain nonlinear equations as for instance the ones in \cite{Krylov1999} cannot be treated right away. In fact, the deterministic solution theory of evolutionary equations in the form discussed in this manuscript is developed for the Hilbert space case only, as of yet, that is. We shall on the other hand emphasise that our considerations do not need any assumptions on the shape of the underlying physical domain. For some results in the deterministic Banach space case we refer to \cite{Weh15}.

(b) The non-autonomous equations discussed here are formulated in a way that the unbounded operator (relation) is independent of the time variable. Again, this is already visible in the deterministic case. A treatment of spdes viewing the spatial operator being time-dependent focussing on explicit (stochastic) partial differential equations can be found in \cite{LR15}. The unbounded operators are then considered to be of elliptic type. This provides a larger class of equations with elliptic and time-dependent spatial operator, which is accessible with \cite{LR15}. On the other hand, that approach makes it difficult to handle the full Maxwell system, see Section \ref{sec:sle}.

\end{remark}

% Use this code if you wish to generate your bibliography with BibTeX;
% please replace first the string "demo" below with the name(s) of
% the BibTeX data base(s) you want to use.
% The resulting bibliography-output (the contents of the .bbl file)
% must be pasted into this file before submission.
% 
% \bibliographystyle{gamm}
% \bibliography{demo}
% 
% Replace the following example bibliography with your references
% before submission:

\end{document}